\documentclass{mfat}
\pagespan{330}{345}

\usepackage{mmap}
\usepackage[utf8]{inputenc}

\theoremstyle{plain}
\newtheorem{theorem}{Theorem}[section]

\newtheorem{proposition}[theorem]{Proposition}

\theoremstyle{assumption}
\newtheorem{assumption}[theorem]{Assumption}
\newtheorem{definition}[theorem]{Definition}
\theoremstyle{remark}
\newtheorem{remark}[theorem]{Remark}

\numberwithin{equation}{section}

\begin{document}

\title[Fredholm theory]
    {Fredholm theory connected with a Douglis-Nirenberg system
of differential equations over $\mathbb{R}^n$}

\author{M. Faierman}
\address{School of Mathematics and Statistics, The University of New South Wales,
 UNSW Sydney, NSW 2052, Australia}
\email{m.faierman@unsw.edu.au}

\date{27/04/2016}

\begin{abstract}
We consider a spectral problem over  $\mathbb{R}^n$ for a
Douglis-Nirenberg system of differential operators under limited
smoothness assumptions  and under the assumption of
parameter-ellipticity in a closed sector $\mathcal{L}$ in the
complex plane with vertex at the origin. We pose the problem in an
$L_p$ Sobolev-Bessel potential space setting, $1 < p < \infty$,
and denote by $A_p$ the operator induced in this setting by the
spectral problem.
 We then derive results pertaining to the Fredholm theory for $A_p$ for values of the spectral parameter
$\lambda$ lying in $\mathcal{L}$ as well as results  pertaining to
the invariance of the Fredholm domain of $A_p$ with $p$.
\end{abstract}

\maketitle

\section{Introduction} \label{S:1}
The Fredholm properties  of elliptic pseudodifferential operators,
as well as systems of such operators, defined over $\mathbb{R}^n$
has been the subject of investigation by various authors over the
past few decades. We refer for example to \cite{F}, \cite{LMc},
\cite{MM},  \cite{R}, and to the references listed therein for
further details. Of particular interest to us are the works
dealing with Douglis-Nirenberg systems (cf. \cite{F}, \cite{MM},
[18], \cite{R}, and  \cite{ZM}) as well as those dealing with
parameter-elliptic operators (cf. \cite{MM}), since, as far as we
are aware from an inspection of the literature, there are no works
dealing with parameter-elliptic Douglis-Nirenberg systems over
$\mathbb{R}^n$ in the classical $L_p$ Sobolev-Bessel potential
space setting, $1 < p < \infty$, when the diagonal operators are
not all of the same order.

In light of what has just been said, let us mention at this point
the paper \cite{Ko} of Kozhevnikov wherein a Douglis-Nirenberg
system of pseudodifferential operators acting over a compact
manifold without boundary is considered. By posing the problem in
a classical setting (as mentioned above) and by introducing the
so-called Kohzevnikov conditions, the author was able to establish
a priori estimates  for solutions as well as various spectral
results. Problems similar to those considered in \cite{Ko} were
also investigated by Denk, Mennicken, and Volevich \cite{DMV}, and
by introducing conditions, which by the use of the Newton polygon
they show to be equivalent to those of Kozhevnikov, they also
establish a priori estimates for solutions as well as various
spectral results.

Motivated by what has been said above, the object of this paper is to derive information concerning the Fredholm properties
of the operator $A_p$ induced in a $L_p$ Sobolev-Bessel potential space setting, $1 < p < \infty$, by a spectral problem over
$\mathbb{R}^n$ for a Douglis-Nirenberg system of differential operators
under limited smoothness assumptions as well as under the assumption of parameter-ellipticity in a closed sector
 $\mathcal{L}$ of the complex plane with vertex at the origin.
And fundamental to our work will be the uniform version of the
Kozhevnikov conditions given in Definition~\ref{D:2.3} below which
will enable us to derive information pertaining to that part of
the Fredholm domain of $A_p$ lying in $\mathcal{L}$.

Turning  to the problem under consideration here, let $N \in
\mathbb{N}$ with $N > 1$ and let $\{s_j\}_1^N$ and $\{t_j\}_1^N$
denote sequences of integers satisfying $s_1 \geq s_2 \geq \cdots
\geq s_N \geq 0$, $t_1 \geq t_2 \geq\cdots \geq t_N \geq 0$, and
put $m_j = s_j + t_j$ for $j = 1,\ldots ,N$. We suppose that $m_1
= m_2=\cdots =m_{k_1}
> m_{k_1+1}=\cdots = m_{k_{d-1}} > m_{k_{d-1}+1} =\cdots = m_{k_d} > 0$, where $k_d = N$,
 and let $\tilde{I}_r$ denote the $(k_r -
k_{r-1})\times (k_r -k_{r-1})$ identity matrix for $r = 1,\ldots
,d$, where $ k_0 = 0$. We will also  use the notation $I_{\ell}$
to denote the $\ell \times \ell$ unit matrix for $\ell \in
\mathbb{N}$.
 Then we shall  be concerned here with the spectral problem
\begin{equation} \label{E:1.1}
A(x,D)u(x) - \lambda u(x)  = f(x) \quad \textrm{in}\quad
\mathbb{R}^n,
\end{equation}
where  $\,u(x) = (u_1(x),\ldots,$ $u_N(x))^T$,
 \;and\; $f(x)= \left(f_1(x),\ldots,f_N(x)\right)^T$\; are\; $N \times 1$ \;matrix functions defined in $\mathbb{R}^n, \;^T$\, denotes transpose,
$A(x,D)$\; is an\; $N \times N$\; matrix operator whose entries\;
$A_{jk}(x,D)$\; are linear differential operators defined on
$\mathbb{R}^n$ of order not exceeding $s_j+t_k$ and defined to be
$0$ if $s_j +t_k <0$. Our assumptions concerning the spectral
problem (1.1) will be made precise in Section~\ref{S:2}.

In Section~\ref{S:2} we make precise the concept of
parameter-ellipticity for the  spectral problem \eqref{E:1.1} and
the conditions under which the problem is dealt with. In
Section~\ref{S:3} we pose the spectral problem in an $L_p$
Sobolev-Bessel potential space setting, $1 < p < \infty$, and
obtain estimates for solutions for various values of the spectral
parameter. Finally in Section~\ref{S:4} we we fix our attention
upon the operator $A_p$ induced in the $L_p$ Sobolev-Bessel
potential space setting just cited by the spectral problem
\eqref{E:1.1}, and using results from Section~\ref{S:3}, we derive
information concerning the Fredholm properties of $A_p$ for
various values
 of the spectral parameter lying in $\mathcal{L}$.


\section{Preliminaries} \label{S:2}
In this section we are going to introduce some terminology,
definitions, and assumptions concerning the spectral problem
(1.1), which we require for our work.

Accordingly, we let $x = (x_1, \ldots, x_n) = (x^{\prime}, x_n)$
denote a generic point in $\mathbb{R}^n$ and use the notation $D_j
= -i\partial/\partial x_j$, $D = (D_1, \ldots , D_n)$, $D^{\alpha}
= D_1^{\alpha_1} \cdots D_n^{\alpha_n}
 = D^{\prime\,\alpha^{\prime}}D_n^{\alpha_n}$, and
$\xi^{\alpha} = \xi_1^{\alpha_1}\ldots \xi_n^{\alpha_n}$ for $\xi
= (\xi_1, \ldots, \xi_n)$ = $(\xi^{\prime}, \xi_n) \in
\mathbb{R}^n$, where $\alpha = (\alpha_1, \ldots, \alpha_n) =
(\alpha^{\prime},\alpha_n)$ is a multi-index whose length $\sum_{j
= 1}^n \alpha_j$ is denoted by $|\alpha|$. Differentiation with
respect to another variable, say $y \in \mathbb{R}^n$, instead of
$x$ will be indicated by replacing $D$ and $D^{\alpha}$ by $D_y$
and $D_y^{\alpha}$, respectively.
 For $1 < p < \infty$, $s \in \mathbb{N}_0 = \mathbb{N}\cup \{\,0 \,\}$,
and $G$ an open set in $\mathbb{R}^n$,  we let $W_p^s(G)$
denote the Sobolev space of order $s$ related to $L_p(G)$ and denote the norm in this
space by $\| \cdot \|_{s,p,G}$, where $\|u\|_{s,p,G} = \left(\sum_{|\alpha| \leq s}
\int_G |D^{\alpha}u(x)|^p \, dx \right)^{1/p}$ for $u \in W^s_p(G)$.
In addition we
shall use norms depending upon a parameter $\lambda \in \mathbb{C} \backslash \{\,0\,
\}$, namely for $1 \leq j \leq N$,\; we let
\[
||| u||| ^{(j)}_{s,p,G} = \|u\|_{s,p,G}
+|\lambda|^{s/m_j}\|u\|_{0,p,G}\quad\textrm{for} \quad u \in
W^s_p(G).
\]

In the sequel we shall  at times deal with the Bessel-potential
space $H^s_p(G)$ for  $0 \geq s \in \mathbb{Z}$ and equipped with
either its ordinary norm $\Vert\,\cdot\,\Vert_{s,p,G}$ or its
parameter dependent norm $|||\,\cdot\,|||^{(j)}_{s,p,G}, 1 \leq j
\leq N$. Here, for $u \in H^s_p(G)$, $\Vert\,u\,\Vert_{s,p,G} =
\Vert\,F^{-1}\langle\,\xi\,\rangle^sFu\,\Vert_{0,p,\mathbb{R}^n}$
(resp. $||| u|||^{(j)}_{s,p,G} =
\Vert\,F^{-1}\langle\xi,\lambda\rangle^s_jFu\|_{0,p,\mathbb{R}^n}$)
if $G=\mathbb{R}^n$, while $\Vert\,u\,\Vert_{s,p,G} =$
inf\,$\Vert\,v\,\Vert_{s,p,\mathbb{R}^n}$ (resp. $|||
u|||^{(j)}_{s,p,G} =
\textrm{inf}\,|||v|||^{(j)}_{s,p,\mathbb{R}^n}$) otherwise, where
the infimum is taken over all $v \in H^s_p(\mathbb{R}^n)$ for
which  $u = v_{\bigl|G}$, $F$ denotes the Fourier transformation
in $\mathbb{R}^n (x \to \xi)$, $\langle\,\xi\,\rangle = (1 +
|\xi|^2)^{1/2}$, $\lambda \in \mathbb{C} \setminus \{\,0\,\}$, and
$\langle\xi,\lambda\rangle_j =
\left(|\xi|^2+|\lambda|^{\frac{2}{m_j}}\right)^{\frac{1}{2}}$ (see
\cite[Section~1]{GK}, \cite[p.~177]{T}). Analogous definitions
also hold for $s > 0$. However
 when $s \geq 0$, then $W^s_p(G)$ and $H^s_p(G)$ coincide algebraically
and their norms, both ordinary and parameter  dependent, are equivalent.
Hence in the sequel, when dealing with the space $H^s_p(G)$ for $s \geq  0$, we shall  suppose that it is equipped
with either the ordinary or parameter dependent norm of $W^s_p(G)$.
Lastly, let
   $\mathbb{R}_+ = \{\,t \in \mathbb{R} | t > 0\,\}$, $\mathbb{R}_- = \{\,t \in \mathbb{R}\,\bigl|\,t < 0\,\}$.

Next for $\ell \in \mathbb{N}_0$,
 we will use the usual notation $C^{\ell}(\mathbb{R}^n)$ to denote the vector space consisting of all those functions $\phi$
which, together with their partial derivatives of order up to $\ell$, are   continuous on $\mathbb{R}^n$.
 In addition we let $C^{\ell}(\overline{\mathbb{R}^n})$ denote the subspace
of $C^{\ell}(\mathbb{R}^n)$ consisting of all those functions $\phi \in C^{\ell}(\mathbb{R}^n)$ for which $\phi$  as well as its partial derivative
of order up to $\ell$ are bounded and uniformly continuous on $\mathbb{R}^n$. Lastly for $\ell \in \mathbb{N}$ we let $C^{\ell,0}(\overline{\mathbb{R}^n})$
denote the subspace of $C^{\ell}(\overline{\mathbb{R}^n})$ for which $|D^{\alpha}\phi(x)| \to 0$ as $|x| \to \infty$ for $1 \leq |\alpha| \leq \ell$,
while for $ \ell = 0$ we let $C^{0,0}(\overline{\mathbb{R}^n})$ denote the subspace of $C^0(\overline{\mathbb{R}^n})$ consisting of those functions $\phi$
for which $\omega_{\phi}(x) \to 0$ as $|x| \to \infty$, where $\omega_{\phi}(x) =$ sup\,$|\phi(x) - \phi(y)|$ for $x,y \in \mathbb{R}^n$, and
where for each $x$ the supremum is taken over those values of $y$ for which $|x-y| \leq 1$.
Note that for $\ell \geq  1, C^{\ell,0}(\overline{\mathbb{R}^n}) \subset C^{0,0}(\overline{\mathbb{R}^n})$.

Turning now to the spectral problem (1.1), (1.2), let us write
\begin{equation*}  A_{jk}(x,D) = \sum_{|\alpha| \leq s_j+t_k}a^{jk}_{\alpha}(x)D^{\alpha} \quad \textrm{for}
 \quad x \in  \mathbb{R}^n \quad
\textrm{and} \quad 1 \leq j,k \leq N.
\end{equation*}
Then as pointed out in \cite{ADN} there is no loss of generality in making the following assumptions.
\begin{assumption} \label{A:2.1}
\rm{It will henceforth be supposed that $t_j > 0$ and $s_j \geq 0$ for $ j = 1,
 \ldots,N$. }
\end{assumption}
\begin{definition} \label{D:2.2}
\rm{We say that the spectral problem \eqref{E:1.1} is minimally smooth if
    for each pair $j,k$,
  $a_{\alpha}^{jk} \in C^{s_j}(\overline{\mathbb{R}^n})$ \; for
  $|\alpha| \leq s_j+t_k$ if $s_j> 0$, while if $s_j = 0$, then
  $a_{\alpha}^{jk} \in L_{\infty}(\mathbb{R}^n)$ for $|\alpha| < t_k$ and
  $a_{\alpha}^{jk} \in C^0(\overline{\mathbb{R}^n})$ for $|\alpha| = t_k$.}
   \end{definition}

For $x,\xi \in \mathbb{R}^n$   let
\[
\mathring{{A}}(x,\xi) = \left(\mathring{{A}}_{jk}(x,\xi)\right)_{j,k=1}^N,
\]
where $\mathring{{A}}_{jk}(x,\xi)$
  consists of those terms in
 $A_{jk}(x,\xi)$  which are just of order $s_j+t_k$.
  Then
 in the sequel we shall also require the following notation. For
 $x,\xi \in \mathbb{R}^n$  let
  \[
 \mathcal{A}^{(r)}_{11}(x,\xi) = \left(\mathring{{A}}_{jk}(x,\xi)\right)_{j,k=1}^{k_r}
 \quad \mathrm{for}\quad 1 \leq r \leq d.
 \]
  In addition we let  $\tilde{I}_{r,0 }=$ diag$(0\cdot\tilde{I}_1,\ldots,0\cdot \tilde{I}_{r-1},\tilde{I}_r)
  \
\textrm{for} \ r=2,\ldots,d$ and $\tilde{I}_{1,0} = \tilde{I}_1$.

  \begin{definition} \label{D:2.3}
 \rm{Let $\mathcal{L}$ be a closed sector in the complex plane with vertex at the origin.
  Then the spectral problem \eqref{E:1.1}
  will be called parameter-elliptic in $\mathcal{L}$
  if it is minimally smooth  and for each $r, 1 \leq r \leq d$,
  \begin{equation} \label{E:2.1}
         \bigl|\, \textrm{det}\left(\mathcal{A}^{(r)}_{11}(x,\xi) - \lambda\,\tilde{I}_{r,0}\right)\bigr| \geq \kappa_r|\xi|^{N_{r-1}}
   \end{equation}
   for   $x$,
 $\xi \in \mathbb{R}^n$, and $\lambda \in \mathcal{L}$ with $\langle\,\xi,\lambda\,\rangle_{k_r} = 1$,
 where the $\kappa_r$ denote positive constants and $N_r = \sum_{j=1}^{k_r}m_j$ for $r \geq 1$ and $N_0 = 0$. }
     \end{definition}

     \begin{remark} \label{R:2.4}
\rm{It follows from the arguments of \cite[Proposition~2.2]{AV}
that when \eqref{E:2.1} holds, then $N_r$ is even if $r = 1$ and
if $r >1$ and $n > 2$.}
\end{remark}

\begin{definition} \label{D:2.5}
    \rm{We say that the spectral problem \eqref{E:1.1} is weakly smooth if it is minimally smooth and in addition
    $a^{jk}_{\alpha} \in C^{0,0}(\overline{\mathbb{R}^n})$ if  $|\alpha| \leq t_k$}
    and to $C^{|\alpha|-t_k,0}(\overline{R^n})$ otherwise, $1 \leq j,k  \leq     N$.
\end{definition}

\section{Some estimates} \label{S:3}
In this section we are going to establish some a priori estimates for solutions of the spectral problem \eqref{E:1.1},
which will be used in the sequel.

Accordingly, let us introduce some further notation. For $G$ an open set in $\mathbb{R}^n$ and $\tau = t$ or $s$
we let  $W^{(\tau)}_p(G) =
 \prod_{k=1}^NW^{\tau_k}_p(G)$ and equip $W^{(\tau)}_p(G)$ with either its ordinary norm,
 $\Vert\,u\,\Vert_{(\tau),p,G} = \sum_{k=1}^N\Vert\,u_k\,\Vert_{\tau_k,p,G}$
or its parameter dependent norm $|||u|||_{(\tau),p,G} = \sum_{k=1}^N|||u_k|||^{(k)}_{\tau_k,p,G}$ for
$u = \left(u_1,\ldots,u_N\right)^T \in W^{(\tau)}_p(G)$. The subspace
$\prod_{k=1}^N\textrm{\r{W}}^{\tau_k}_p(G)$ is denoted by $\textrm{\r{W}}^{(\tau)}_p(G)$. In addition we let $H^{(-s)}_p(G) = \prod_{k=1}^NH^{-s_k}_p(G)$
  and equip this space with either its ordinary norm
  $\Vert\cdot\Vert_{(-s),p,G}$    or its parameter dependent norm
$|||\cdot|||_{(-s),p,G}$, which are defined in an analogous
manner to the way they were defined for $W^{(\tau)}_p(G)$.

\begin{proposition} \label{P:3.1}
 Suppose that the spectral problem \eqref{E:1.1} is minimally smooth. Suppose also
 that $\lambda \in \mathbb{C}$ with $|\lambda| \geq \lambda^{\sharp}$ for some $\lambda^{\sharp} \in \mathbb{R}_+$, that
  $u \in W^{(t)}_p(\mathbb{R}^n)$, and that $f$ is defined by \eqref{E:1.1}.
 Then $f \in H^{(-s)}_p(\mathbb{R}^n)$ and $|||f|||_{(-s),p,\mathbb{R}^n}  \leq C|||u|||_{(t),p,\mathbb{R}^n}$,
 where the constant $C$ does not depend upon $u$ and $\lambda$.
\end{proposition}

\begin{proof}  Let $\langle\,\cdot,\cdot\,\rangle_{\mathbb{R}^n}$
denote the pairing between $H^{(-s)}_p(\mathbb{R}^n)$ and its
adjoint space $W^{(s)}_{p^{\prime}}(\mathbb{R}^n)$, where
$p^{\prime} = p/(p-1)$ and both spaces are equipped with their
parameter dependent norms (see \cite[Section~1]{GK} and
\cite[Theorem~2.6.1, p.~198]{T}). Then for $ u \in
W^{(t)}_p(\mathbb{R}^n)$ and $\zeta  =
\left(\zeta_1,\ldots,\zeta_N\right)^T \in
C^{\infty}_0(\mathbb{R}^n)^N$,
\begin{equation*}
  \langle\left(A(x,D)u - \lambda\,u\right),\zeta\rangle_{\mathbb{R}^n} = \sum_{j=1}^N\left(\sum_{k=1}^NA_{jk}(x,D)u_k - \lambda\,\delta_{jk}u_k\right)(\overline{\zeta_j}),
\end{equation*}

where $\delta_{jk}$ denotes the Kronecker delta, and for each $j$
\[
  \left(\sum_{k=1}^NA_{jk}(x,D)u_k - \lambda\,\delta_{jk}u_k\right)(\overline{\zeta_j})
  \]
denotes the value of the distribution $\sum_{k=1}^NA_{jk}(x,D)u_k - \lambda\,\delta_{jk}u_k$ at $\overline{\zeta_j}$.

Let us fix our attention upon a particular pair $j,k$. Then
\[
\begin{aligned}
 \left(A_{jk}(x,D)u_k - \lambda\,\delta_{jk}u_k\right)(\overline{\zeta_j}) &
 = \sum_{\substack{t_k\leq|\alpha|\leq s_j+t_k\\|\beta|=t_k}}
 \left(D^{\beta}u_k,D^{\alpha-\beta}\overline{a^{jk}_{\alpha}(x)}\zeta_j\right)_{\mathbb{R}^n}
 \\
 &
 \quad
  + \left(\sum_{|\alpha| < t_k}a^{jk}_{\alpha}(x)D^{\alpha}u_k,\zeta_j\right)_{\mathbb{R}^n}
   - \lambda\,\delta_{jk}\left(u_k,\zeta_j\right)_{\mathbb{R}^n},
 \end{aligned}
\]
where $\left(\cdot,\cdot\right)_{\mathbb{R}^n}$ denotes the pairing between $L_p(\mathbb{R}^n)$
and its adjoint space $L_{p^{\prime}}(\mathbb{R}^n)$. Hence
\[
\begin{aligned}
 \bigl|\left(A_{jk}(x,D)u_k - \lambda\,\delta_{jk}u_k\right)(\overline{\zeta_j})\bigr| & \leq
 C\Bigl(\Vert\,u_k\,\Vert_{t_k,p,\mathbb{R}^n}\Vert\,\zeta_j\,\Vert_{s_j,p^{\prime},\mathbb{R}^n}
 \\
 & \quad +
 |\lambda|^{t_k/m_k}|\lambda|^{s_j/m_j}\Vert\,u_k\,\Vert_{0,p,\mathbb{R}^n} \Vert\,\zeta_j\,
 \Vert_{0,p^{\prime},\mathbb{R}^n}\Bigr) \\
 & \leq C|||u_k|||^{(k)}_{t_k,p,\mathbb{R}^n}|||\zeta_j|||^{(j)}_{s_j,p^{\prime},\mathbb{R}^n},
 \end{aligned}
\]
where the constant $C$ does not depend upon $u_k, \zeta_j$, and $\lambda$.
It now follows from the foregoing results  that
\begin{equation*}
 \bigl|\langle\left(A(x,D) - \lambda\,I_N\right)u,\zeta\rangle_{\Omega}\bigr| \leq C|||u|||_{(t),p,\mathbb{R}^n}|||\zeta|||_{(s),p^{\prime},\mathbb{R}^n},
\end{equation*}
and the assertion of the proposition follows immediately from this last result.
\end{proof}

We now turn to the main results of this section.
\begin{proposition} \label{P:3.2}
 Suppose that the spectral problem \eqref{E:1.1} is parameter-elliptic in $\mathcal{L}$.
 Suppose also that $u \in W^{(t)}_p(\mathbb{R}^n)$ and
 that $f$ is defined by \eqref{E:1.1}. Then there exists the constant
 $\lambda^{\prime} = \lambda^{\prime}(p) > 0$ such that for $\lambda \in \mathcal{L}$ with $|\lambda| \geq \lambda^{\prime}$, the a priori estimate
 \begin{equation} \label{E:3.1}
  |||u|||_{(t),p,\mathbb{R}^n} \leq C|||f|||_{(-s),p,\mathbb{R}^n}
 \end{equation}
holds, where the constant $C$ does not depend upon $u$ and $\lambda$.
 \end{proposition}
\begin{proposition} \label{P:3.3}
 Suppose that the spectral problem \eqref{E:1.1} is parameter-elliptic in $\mathcal{L}$. Then there exists the constant
 $\lambda^0 = \lambda^0(p) > 0$ such that for $\lambda \in \mathcal{L}$ with $|\lambda| \geq \lambda^0$, the spectral problem  has a unique solution
 $u \in W^{(t)}_p(\mathbb{R}^n)$ for every $f \in H^{(-s)}_p(\mathbb{R}^n)$,  and the a priori estimate
 \[
  |||u|||_{(t),p,\mathbb{R}^n} \leq C|||f|||_{(-s),p,\mathbb{R}^n}
 \]
holds, where the constant $C$ does not depend upon $f$ and $\lambda$.
\end{proposition}

The proofs of Propositions~\ref{P:3.2} and \ref{P:3.3} can be
achieved by modifying the arguments given in the proofs of
Theorems~4.1 and 5.1 of \cite{AV}. Furthermore, since the proof of
Proposition~\ref{P:3.3} is somewhat similar to that of
Proposition~\ref{P:3.2}, we will restrict ourselves to the proof of
this latter proposition.

Before beginning the proof, let us now present some results which we require below.
\begin{proposition} \label{P:3.4}
Suppose that the  spectral problem \eqref{E:1.1} is
parameter-elliptic in $\mathcal{L}$ and that $f \in
H^{(-s)}_p(\mathbb{R}^n)$.
 Suppose also that $x^0 \in \mathbb{R}^n$. Then there exits the
constant $\lambda_1 = \lambda_1(p) > 0$ such that for $\lambda \in \mathcal{L}$ with $|\lambda| \geq \lambda_1$, the equation
\[
 \textrm{\r{A}}(x^0,D)u(x) - \lambda\,u(x) = f(x) \quad\textrm{for}\quad x \in \mathbb{R}^n
 \]
has a unique solution $u \in W^{(t)}_p(\mathbb{R}^n)$ and the a priori estimate
\[
|||u|||_{(t),p,\mathbb{R}^n} \leq C_1|||f|||_{(-s),p,\mathbb{R}^n},
\]
holds, where the constant $C_1$ does not depend upon $x^0, f,$ and $\lambda$.
\end{proposition}

\begin{proof}
It follows from Definition
\ref{D:2.3} and \cite{DMV} that there exists the constant
$\lambda_0 > 0$ such that for $\lambda \in \mathcal{L}$  with
$|\lambda| \geq \lambda_0, \xi \in \mathbb{R}^n$, and $\alpha$ a
multi-index whose entries are either $0$ or $1$,
\begin{align*}
 \bigl|\,\textrm{det}\,\left(\textrm{\r{A}}(x^0,\xi) - \lambda\,I_N\right)\bigr| &
 \geq C_0\prod_{j=1}^N\langle\,\xi,\lambda\rangle_j^{m_j} \quad \textrm{and}
 \\
 \bigl|\,\xi^{\alpha}D^{\alpha}_{\xi}\tilde{a}_{jk}(x^0,\xi,\lambda)\,\bigr| &
 \leq
 C_0^{\prime}\langle\,\xi,\lambda\,\rangle^{-s_k}_k\langle\,\xi,\lambda\rangle^{-t_j}_j,
\end{align*}
where the constants $C_0$ and $C_0^{\prime}$ do not depend upon
$\alpha, x^0, \xi, \lambda, j, k$, and where we have written
$\left(\textrm{\r{A}}(x^0,\xi) - \lambda\,I_N\right)^{-1} =
\left(\tilde{a}_{jk}(x^0,\xi,\lambda\right)_{j,k=1}^N$. The
assertions of the proposition now follow from the same arguments
as those used in the proof of \cite[Proposition~3.2]{DF}.
\end{proof}
\begin{proposition} \label{P:3.5}
Suppose that the spectral problem \eqref{E:1.1} is parameter-elliptic in $\mathcal{L}$.
Then there exist
the constants $r_0 = r_0(p)$ and $\lambda_2 = \lambda_2(p)$ in $\mathbb{R}_+$ such that for each
$x^0 \in \mathbb{R}^n$ one can find a   neighbourhood $V$ of this point with diam\,$V \leq r_0$  for which
the a priori estimate
\[
 |||u|||_{(t),p,\mathbb{R}^n} \leq C|||A(x,D)u - \lambda\,u|||_{(-s),p,\mathbb{R}^n}
\]
holds for each $u \in W^{(t)}_p(\mathbb{R}^n)$ and  $\lambda   \in \mathcal{L}$  for which
 supp\,$u \subset V$ and $|\lambda| \geq \lambda_2$,
 where the constant
$C$ does not depend upon $u$ and  $\lambda$, and where diam denotes diameter and supp denotes support.
\end{proposition}

\begin{proof}
 To begin with let us choose $r_0$ so that
for $x,y \in \mathbb{R}^n$ with $|x-y| \leq r_0$ we have for each triple $(j,k,\alpha)$ with  $1 \leq j,k \leq N$ and $|\alpha| = s_j+t_k, \,
|a^{jk}_{\alpha}(x) - a^{jk}_{\alpha}(y)| \leq 1/8C_1\delta_1$,
where $C_1$  denotes the constant of Proposition
\ref{P:3.4}, $\delta_1 = N^2$\,max\,$\bigl\{\frac{(s_j+t_k)!}{s_j!}n_{jk}\,\bigr\}_{j,k=1}^N$ and
 $n_{jk}$ denotes the number of distinct multi-indices $\alpha$ for which $|\alpha| = s_j+t_k$.

 Next let  $x^0 \in \mathbb{R}^n$ and let $V \subset \mathbb{R}^n$ be a neighbourhood of $x^0$ with diam\,$V \leq   r_0$.
 Also let $u \in W^{(t)}_p(\mathbb{R}^n)$ such that supp\,$u \subset V$. Then bearing in mind the proof
 of Proposition \ref{P:3.1} as well as Proposition \ref{P:3.4}, we see that
  \[
  |||u|||_{(t),p,V} \leq C_1|||\left(\textrm{\r{A}}(x^0,D) - \lambda\,I_N\right)u|||_{(-s),p,\mathbb{R}^n}
 \]
for $\lambda$ and $C_1$ satisfying the conditions
cited in       Proposition \ref{P:3.4}. Hence it follows from arguments similar to those used in the proof of
\cite[Proposition 4.1]{DF} that
there exists the constant $\lambda_2 = \lambda_2(p) > 0$ such that
 \[
 |||u|||_{(t),p,V} \leq C_1^{\prime}|||\left(A(x,D)
 -\lambda\,I_N\right)u|||_{(-s),p,\mathbb{R}^n}\quad
 \textrm{for}\quad  \lambda \in \mathcal{L} \quad \textrm{with} \quad  |\lambda| \geq \lambda_2,
 \]
 where the constant $C_1^{\prime}$ has the same properties as the constant $C_1$.
 On the other hand, if we refer to the proof of Proposition \ref{P:3.1} for notation and argue as in that proof,
 then we obtain for $\zeta \in C^{\infty}_0(\mathbb{R}^n)^N$ and for values of $\lambda$ just cited,
 \begin{equation*}
  \langle\,\left(A(x,D) - \lambda\,I_N\right)u,\zeta\,\rangle_{\mathbb{R}^n} = \langle\,\left(A(x,D) - \lambda,I_N\right)u,\chi\,\zeta\rangle_V,
 \end{equation*}
where $\chi \in C^{\infty}_0(\mathbb{R}^n)$ such that $\chi(x) =
1$ for $x \in$ supp\,$u$ and supp\,$\chi \subset V$. Hence if we
take into account \cite[Proposition~3.5, p.~109]{CP},
\cite[Theorems~ 4.3.2.1, p.~317 and 4.8.2, p.~332]{T}, and the
fact that $\chi\,\zeta\in \textrm{\r{W}}^{(s)}_{p^{\prime}}(V)$,
it follows that
\[
 \begin{aligned}
  \bigl|\langle\,\left(A(x,D) - \lambda\,I_N\right)u,\zeta\,\rangle_{\mathbb{R}^n}\bigr| &\leq |||\left(A(x,D) -
  \lambda\,I_N\right)u|||_{(-s),p,V}|||\chi\,\zeta|||_{(s),p^{\prime},V} \\
  &\leq C_0|||\left(A(x,D) -\lambda\,I_N\right)u|||_{(-s),p,V}|||\zeta||_{(s),p^{\prime},\mathbb{R}^n},
 \end{aligned}
 \]
where
the constant $C_0$ does not depend upon $u$ and $\lambda$.
This completes the proof of the Proposition.
\end{proof}

  \noindent\textit{Proof of Proposition~\ref{P:3.2}.}
  For $t > 0$ let $Q(t)$ denote the open cube in $\mathbb{R}^n$ with centre at the origin and with sides parallel to the coordinate axes
and of length $2t$. Also let $\zeta, \chi \in C^{\infty}_0(\mathbb{R}^n)$ such that $\zeta(x) = 1$ for $x \in Q(1/2)$ and $\zeta(x) = 0$
 for $x \in \mathbb{R}^n \setminus Q(3/4)$, while $\chi(x) = 1$ for $x \in Q(13/16)$ and $\chi(x) = 0$ for $x \in \mathbb{R}^n \setminus Q(14/16)$.
 Then for $d > 0$ and $\gamma \in \mathbb{Z}^n$, we let $Q_{\gamma,d} =
 \bigl\{\,x \in \mathbb{R}^n\,\bigl|\,|x-d\gamma|\in Q(d)\,\bigr\}$,  $\zeta_{\gamma,d}(x) = \zeta(\left(x-d\gamma)/d\right)$,
 and $\chi_{\gamma,d}(x) = \chi\left((x-d\gamma)/d\right)$.

 Next let $d_0$ denote a constant satisfying $0 < d_0 < r_0/4n^{1/2}$ (see Proposition~\ref{P:3.5}). Then
 we shall henceforth fix $d \leq d_0$ and let $\{\,\gamma_j\,\}_1^{\infty}$ denote an enumeration of the members of $\mathbb{Z}^n$
 and put $Q_j = Q_{\gamma_j,d}, \zeta_j(x) = \zeta_{\gamma_j,d}(x)$ and $\chi_j(x) = \chi_{\gamma_j,d}(x)$.
Hence if for $j \geq 1$ we let $\eta_j(x) = \zeta_j(x)/\sum_{j=1}^{\infty}\zeta_j(x)$, then $\{\,Q_j\,\}_1^{\infty}$ is an open covering of
$\mathbb{R}^n$ and $\bigl\{\,\eta_j(x)\,\bigr\}_1^{\infty}$ a partition of
unity subordinate to this covering. Note also that for any $x \in \mathbb{R}^n$, $x$ can lie in at most $2^n$ of the $Q_j$.

Bearing in mind Proposition \ref{P:3.5} let $u \in W^{(t)}_p(\mathbb{R}^n)$ and $\lambda \in \mathcal{L}$ with $|\lambda| \geq \lambda_2$.
Then  we have
\[
 |||u|||_{(t),p,\mathbb{R}^n} \leq
   N^{1/p^{\prime}}\left(\sum_{j=1}^N|||\sum_{\ell \geq 1}\eta_{\ell}u_j|||_{t_j,p,\mathbb{R}^n}^{(j)\,p}\right)^{1/p},
 \]
 which leads to the inequality
 \begin{equation*}
  |||u|||_{(t),p,\mathbb{R}^n} \leq
  2^{n+1}N\left(\sum_{\ell \geq 1}|||\eta_{\ell}u|||^p_{(t),p,Q_{\ell}}\right)^{1/p},
 \end{equation*}
 and hence it follows from Proposition \ref{P:3.5} and  its proof that
 \begin{equation} \label{E:3.2}
|||u|||_{(t),p,\mathbb{R}^n} \leq C
                                \left(\sum_{\ell \geq 1}|||\bigl(A(x,D) - \lambda\,I_N\bigr)\eta_{\ell}u|||_{(-s),p,Q_{\ell}}^p\right)^{1/p},
\end{equation}
where here and for the remainder of this proof $C$ denotes a generic constant which may vary from inequality to inequality,
but in all cases it does not depend upon $u, \lambda$ and $\ell$ (see below). Note that here we have used the fact that under our present definitions of $V$ and
$\chi$, we can take the constant $C$ of Proposition  \ref{P:3.5} to be also independent of $x^0$.

Let us firstly fix our attention upon a particular $\ell \geq 1$. Then observing that
\[
\begin{aligned}
&  |||\left(A(x,D) - \lambda\,I_N\right)\eta_{\ell}u -
\eta_{\ell}\left(A(x,D)-
 \lambda\,I_N\right)u|||_{(-s),p,Q_{\ell}}  \\
& \qquad \leq  |||\left(A(x,D) - \lambda\,I_N\right)\eta_{\ell}u
 - \eta_{\ell}\left(A(x,D) - \lambda\,I_N\right)u|||_{(-s),p,\mathbb{R}^n},
 \end{aligned}
\]
and referring to the proof of Proposition \ref{P:3.1} for notation, we can argue as in that proof to show that
\[
 \begin{aligned}
  &\bigl|\langle\,\left(\bigl(A(x,D) - \lambda\,I_N\bigr)\eta_{\ell}u - \eta_{\ell}\bigl(A(x,D) -
  \lambda\,I_N\bigr)u\right),\zeta\rangle_{\mathbb{R}^n}\bigr| \\
  & \qquad \leq  C\sum_{j=1}^N\sum_{k=1}^N\left(\Vert\,\chi_{\ell}u_k\,\Vert_{t_k-1,p,Q_{\ell}}\Vert\,\chi_{\ell}\zeta_j\,\Vert_{s_j,p^{\prime},Q_{\ell}} +
   \Vert\,\chi_{\ell}u_k\,\Vert_{t_k,p,Q_{\ell}}\Vert\,\chi_{\ell}\zeta_j\Vert_{s_j-1,p^{\prime},Q_{\ell}}\right).
   \end{aligned}
\]
Hence it follows from \cite[Proposition~2.2]{ADF} and the fact
that $\chi_{\ell}$ has  support in $Q_{\ell}$ that
\begin{equation} \label{E:3.3}
 \begin{aligned}
  & \bigl|\langle\,\left(\bigl(A(x,D) - \lambda\,I_N\bigr)\eta_{\ell}u - \eta_{\ell}\left(A(x,D) -
  \lambda\,I_N\right)u\right), \zeta\,\rangle_{\mathbb{R}^n}\bigr|  \\
&  \qquad  \leq
C|\lambda|^{-1/m_1}|||u|||_{(t),p,Q_{\ell}}|||\zeta|||_{(s),p^{\prime},Q_{\ell}}.
 \end{aligned}
\end{equation}
Thus we conclude from \eqref{E:3.3} that
\[
 \begin{aligned}
 & |||\left(A(x,D) - \lambda\,I_N\right)\eta_{\ell}u|||_{(-s),p,\mathbb{R}^n} \leq \\
 & \qquad  |||\eta_{\ell}\left(A(x,D) - \lambda\,I_N\right)u|||_{(-s),p,\mathbb{R}^n}   +
    C|\lambda|^{-1/m_1}|||u|||_{(t),p,Q_{\ell}},
  \end{aligned}
\]
and hence  that
\begin{equation} \label{E:3.4}
\begin{aligned}
 &\sum_{\ell\geq1}||\left(A(x,D) - \lambda\,I_N\right)\eta_{\ell}u|||_{(-s),p,\mathbb{R}^n}^p \\
 &  \qquad \leq  2^p\sum_{\ell\geq 1}|||\eta_{\ell}\left(A(x,D) -
 \lambda\,I_N\right)u|||_{(-s),p,\mathbb{R}^n}^p  +
    2^pC^p|\lambda|^{-p/m_1}\sum_{\ell\geq1}|||u|||_{(t),p,Q_{\ell}}^p.
\end{aligned}
\end{equation}

We are now going to use \eqref{E:3.2} and \eqref{E:3.4} to complete the proof of the proposition. To this end,
some further preparation is required. Accordingly, let us recall that we have so far equipped the space
$W^{s_j}_{p^{\prime}}(\mathbb{R}^n), 1 \leq  j \leq N$, with the parameter dependent norm $|||\cdot|||_{s_j,p^{\prime},\mathbb{R}^n}$. For our purposes
it will be convenient now to introduce an equivalent norm, namely the norm $|||\cdot|||_{s_j,p^{\prime},\mathbb{R}^n}^{\prime}$ defined by
\begin{equation*}
 |||v|||^{\prime}_{s_j,p^{\prime},\mathbb{R}^n} =  \\
 \left(\Vert\,v\,\Vert^{p^{\prime}}_{s_j,p^{\prime},\mathbb{}{R}^n} + \bigl[(1 +
 |\lambda|^{s_j/m_j})^{p^{\prime}}-1\bigr]\Vert\,v\,\Vert^{p^{\prime}}_{0,p^{\prime},\mathbb{R}^n}\right)^{1/p^{\prime}}
\end{equation*}
for $v \in W^{s_j}_{p^{\prime}}(\mathbb{R}^n)$.
Note that by the equivalence of these two norms we mean
that there are constants $C_1$ and $C_2$, not depending upon $v$ and $\lambda$, such that
$|||v|||_{s_j,p^{\prime},\mathbb{R}^n} \leq C_1|||v|||_{s_j,p^{\prime},\mathbb{R}^n}^{\prime}$ and
$|||v|||^{\prime}_{s_j,p^{\prime},\mathbb{R}^n} \leq C_2|||v|||_{s_j,p^{\prime},\mathbb{R}^n}$.

Supposing now that $\lambda$ is fixed and
$W^{s_j}_{p^{\prime}}(\mathbb{R}^n)$ is equipped with the norm
$|||\cdot|||^{\prime}_{s_j,p^{\prime},\mathbb{R}^n}$, let $\nu_j$
denote the number of distinct  multi-indices $\alpha$ satisfying
$0 \leq |\alpha| \leq s_j$. Then to each $v \in
W^{s_j}_{p^{\prime}}(\mathbb{R}^n)$ we can associate the vector
$\mathcal{P}_{j,p^{\prime}}v \in L_{p^{\prime}}(\Omega)^{\nu_j}$,
where $\mathcal{P}_{j,p^{\prime}}v$ denotes the $\nu_j$ -- vector
with components
$\bigl\{\,\Lambda_{j,\alpha}D^{\alpha}\,\bigr\}_{|\alpha| \leq
s_j}$, where $\Lambda_{j,\alpha} = 1$ if $|\alpha| \geq 1$ and
equals $1 + |\lambda|^{s_j/m_j}$ otherwise, and where the
components of $\mathcal{P}_{j,p^{\prime}}v$ are arranged so that
if $\Lambda_{j,\alpha}D^{\alpha}v$ is the $\ell$-th component of
this vector and $\Lambda_{j,\beta}D^{\beta}v$ the $(\ell+1)$-th
component, then  $|\alpha| \leq |\beta|$. Hence if we let
$\mathcal{H}_{j,p^{\prime}}$ denote the closed subspace of
$L_{p^{\prime}}(\Omega)^{\nu_j}$ spanned by the vectors
$\mathcal{P}_{j,p^{\prime}}v$ for $v \in
W^{s_j}_{p^{\prime}}(\mathbb{R}^n)$, then
$\mathcal{H}_{j,p^{\prime}}$ is isometrically isomorphic to
$W^{s_j}_{p^{\prime}}(\mathbb{R}^n)$.

Turning next to the space $W^{(s)}_{p^{\prime}}(\mathbb{R}^n)$, we
have so far equipped this space with the parameter dependent norm
$|||\cdot|||_{(s),p^{\prime},\mathbb{R}^n}$. Let us also equip
$W^{(s)}_{p^{\prime}}(\mathbb{R}^n)$ with the norm  \break
$|||\cdot|||_{(s),p^{\prime},\mathbb{R}^n}^{\prime}$, which is
equivalent to $|||\cdot|||_{(s),p^{\prime},\mathbb{R}^n}$, defined
by $|||v|||^{\prime\,p^{\prime}}_{(s),p^{\prime}\,\mathbb{R}^n} =
$ \break $\sum_{j=
1}^N\left(|||v_j|||_{s_j,p^{\prime},\mathbb{R}^n}^{\prime}\right)^{p^{\prime}}$
for $v = (v_1,\ldots,v_N)^T \in
W^{(s)}_{p^{\prime}}(\mathbb{R}^n)$. Note that the norms  \break
$|||\cdot|||_{(s),p^{\prime},\mathbb{R}^n}$ and
$|||\cdot|||^{\prime}_{(s),p^{\prime}\mathbb{R}^n}$ induce
equivalent norms on the adjoint space of
$W^{(s)}_{p^{\prime}}(\mathbb{R}^n)$, $H^{(-s)}_p(\mathbb{R}^n)$,
and hence if we denote by
$|||\cdot|||^{\prime}_{(-s),p,\mathbb{R}^n}$ the norm induced on
$H^{(-s)}_p(\mathbb{R}^n)$ by the norm
$|||\cdot|||^{\prime}_{(s),p^{\prime},\mathbb{R}^n}$, then
$|||\cdot|||^{\prime}_{(-s),p,\mathbb{R}^n}$ and
$|||\cdot|||_{(-s),p,\mathbb{R}^n}$ are equivalent norms on this
space. Note also that when $W^{(s)}_{p^{\prime}}(\mathbb{R}^n)$ is
equipped with the norm
$|||\cdot|||^{\prime}_{(s),p^{\prime},\mathbb{R}^n}$, then
$W^{(s)}_{p^{\prime}}(\mathbb{R}^n)$ is isometrically isomorphic
to $\mathcal{H}_{p^{\prime}}$, the closed subspace of
$L_{p^{\prime}}(\mathbb{R}^n)^{\sum_{j=1}^N\nu_j}$ defined by
$\mathcal{H}_{p^{\prime}} =
\prod_{j=1}^N\mathcal{H}_{j,p^{\prime}}$. Hence it follows from
\cite[Theorem~3.8, p.~49]{Ad} that for each $u \in
H^{(-s)}_p(\mathbb{R}^n)$, there is a $f \in
L_p(\mathbb{R}^n)^{\sum_{j=1}^N\nu_j}$ such that
$|||u|||^{\prime}_{(-s),p,\mathbb{R}^n} = \Vert\,f\,\Vert$, where
$\Vert\cdot\Vert$ denotes the norm in
$L_p(\mathbb{R}^n)^{\sum_{j=1}^N\nu_j}$.

Fixing our attention upon \eqref{E:3.4} and referring to the proof
of Proposition~\ref{P:3.1} for notation, we have for $\ell \geq
1$,
\begin{equation*}
 \begin{aligned}
  |||\eta_{\ell}\left(A(x,D) - \lambda\,I_N\right)u|||_{(-s),p,\mathbb{R}^n} &\leq C|||\eta_{\ell}\left(A(x,D) - \lambda\,I_N\right)u|||^{\prime}_{(-s),p,\mathbb{R}^n} \\
  &= C\,\sup\,\bigl|\langle\,\left(A(x,D) - \lambda\,I_N\right)u,\eta_{\ell}\zeta\,\rangle_{\mathbb{R}^n}\bigr|  \\
  &= C\sup\,\bigl|\sum_{j=1}^N\sum_{k=1}^{\nu_j}\left(f_{j,k},\Lambda_{j,\alpha}D^{\alpha}\eta_{\ell}\zeta_j\right)_{\mathbb{R}^n}\bigr|
   \end{aligned}
\end{equation*}
for some $f \in L_p(\mathbb{R}^n)^{\sum_{j=1}^N\nu_j}$, where the
supremum is taken over the set $\bigl\{\,\zeta \in
C^{\infty}_0(\mathbb{R}^n)^N \,\bigl|\,  \break
|||\zeta|||^{\prime}_{0,p,\mathbb{R}^n} = 1\,\bigr\}$, and where
we have written $f = (f_1,\ldots,f_N)^T$,
 $f_j = (f_{j,1},\ldots,f_{j,\nu_j})^T$,  and $\Lambda_{j,\alpha}D^{\alpha}\eta_{\ell}\zeta_j$ denotes the $k$-th component
 of $\mathcal{P}_{j,p^{\prime}}\eta_{\ell}\zeta_j$.
 Hence it follows that
\begin{equation} \label{E:3.5}
\begin{aligned}
 \sum_{\ell \geq 1}|||\eta_{\ell}\left(A(x,D) - \lambda\,I_N\right)u|||^p_{(-s),p,\mathbb{R}^n} &\leq
 C^p\sum_{\ell \geq 1}\sum_{j=1}^N\sum_{k=1}^{\nu_j}\Vert\,f_{j,k}\,\Vert^p_{0,p,Q_{\ell}} \\
 &\leq C^{\prime}|||\left(A(x,D) - \lambda\,I_N\right)|||^p_{(-s),p,\mathbb{R}^n},
\end{aligned}
\end{equation}
where the constant $C^{\prime}$ has the same properties as the constant $C$.
Hence it follows from equations \eqref{E:3.2}, \eqref{E:3.4}, and \eqref{E:3.5} that we may choose $\lambda^{\prime}(p)$ so that the assertion of the proposition holds.
\hfill{$\square$}

\medskip

To terminate the work of this section, we now present a result
which will be used in the sequel. Here, for $\ell \in \mathbb{N}$,
we use the notation $(t+\ell)$ to denote the multi-index whose
entries are $t_j+\ell, j = 1,\ldots,N$. The multi-indices
$(-s+\ell), (-r_j+\ell-1/p)$ are defined analogously.
\begin{proposition} \label{P:3.6}
 Suppose that the spectral problem \eqref{E:1.1} is parameter-elliptic in $\mathcal{L}$ and
   that $\ell_0 \in \mathbb{N}$. Suppose also  that
   for each pair $j,k, 1 \leq j,k \leq N$,
 $a^{jk}_{\alpha} \in C^{s_j+\ell_0}(\overline{\mathbb{R}^n})$ for $|\alpha| \leq s_j+t_k$.
 Lastly suppose that $u \in W^{(t)}_p(\mathbb{R}^n), \lambda \in \mathcal{L} \setminus \{\,0\,\}$, and that $f$ is defined by
 \eqref{E:1.1}. Consequently,
 if $f \in H^{(-s+\ell_0)}_p(\mathbb{R}^n)$, then $u \in W^{(t+\ell_0)}_p(\mathbb{R}^n)$.
\end{proposition}

\begin{proof}
Referring to the proof of Proposition~\ref{P:3.2} for terminology,
it follows from \eqref{E:3.1} and from arguments similar to those
used in the proof of Proposition~\ref{P:3.1} that for $\ell \geq
1$,
\begin{equation} \label{E:3.6}
\begin{aligned}
\Vert\,\eta_{\ell}u\,\Vert_{(t),p,Q_{\ell}} & \leq
C\bigl(\Vert\,\bigl(A(x,D) -
\lambda\,I_N\bigr)\eta_{\ell}u\,\Vert_{(-s),p,Q_{\ell}}
\\
& \quad  +
\Vert\,\eta_{\ell}u\,\Vert_{(t-1),p,Q_{\ell}}\bigr),
\end{aligned}
\end{equation}
where here and for the remainder of this proof, $C$ denotes a generic constant which may vary from inequality to inequality,
but in all cases  it does not depend upon $\ell, u$,
 and $\lambda$. By employing arguments similar to those used in the proof of Proposition~\ref{P:3.1},
it is not difficult to verify that $\left(A(x,D) -
\lambda\,I_N\right)\eta_{\ell}u \in H^{(-s+1)}_p(Q_{\ell})$. Hence
for $\ell \geq  1$ we can apply the differential quotient method
to \eqref{E:3.6} as in \cite[Proof of Theorem~3.1, p.~123]{LM}
(see also \cite{Ag}) to deduce that $\eta_{\ell}u \in
W^{(t+1)}_p(Q_{\ell})$ and
\begin{equation} \label{E:3.7}
\Vert\,\eta_{\ell}u\,\Vert_{(t+1),p,Q_{\ell}} \leq C\left(\Vert\,\bigl(A(x,D) - \lambda\,I_N\bigr)\eta_{\ell}u\,\Vert_{(-s+1),p,Q_{\ell}}
+ \Vert\,\eta_{\ell}u\,\Vert_{(t),p,Q_{\ell}}\right).
\end{equation}

We conclude from the foregoing results and from arguments similar to those used in the proof of Proposition \ref{P:3.2}
that $u \in W^{(t+1)}_p(B(r))$ for every $r > 0$ where $B(r) = \bigl\{x \in \mathbb{R}^n\,\bigl|
|x| < r\,\bigr\}$. Let us now show that $u \in W^{(t+1)}_p(\mathbb{R}^n)$ and that
\begin{equation} \label{E:3.8}
 \begin{aligned}
  \Vert\,u\,\Vert_{(t+1),p,\mathbb{R}^n} & \leq C\bigl(\Vert\,\bigl(A(x,D) - \lambda\,I_N\bigr)u\,
  \Vert_{(-s+1),p,\mathbb{R}^n}
  \\
  & \quad +
  \Vert\,u\,\Vert_{(t),p,\mathbb{R}^n}\bigr).
   \end{aligned}
\end{equation}
Accordingly, it is clear that
\begin{equation*}
\Vert\,u\,\Vert_{(t+1),p,\mathbb{R}^n} \leq
N^{1/p^{\prime}}\left(\sum_{j=1}^N\Vert\,\sum_{\ell \geq
1}\eta_{\ell}u_j\,\Vert^p_{t_j+1,p,\mathbb{R}^n}\right)^{1/p},
\end{equation*}
which leads to the inequality
\begin{equation} \label{E:3.9}
\Vert\,u\,\Vert_{(t+1),p,\mathbb{R}^n} \leq
2^{n+1}N\left(\sum_{\ell \geq 1}\Vert\,\eta_{\ell}u_j\,\Vert^p_{(t+1),p,Q_{\ell}}\right)^{1/p}.
\end{equation}
Furthermore, by appealing to \eqref{E:3.7} and by arguing in a manner similar to that in the proof of Proposition \ref{P:3.2}, we can also show that
\begin{equation} \label{E:3.10}
\begin{aligned}
 \sum_{\ell \geq 1}\Vert\,\eta_{\ell}u\,\Vert^p_{(t+1),p,Q_{\ell}} &\leq
 C\sum_{\ell \geq 1}\left(\Vert\,\eta_{\ell}\bigl(A(x,D) - \lambda\,I_N\bigr)u\,\Vert^p_{(-s+1),p,Q_{\ell}}  \right.
 \\
 & \quad + \left. \Vert\,u\,\Vert^p_{(t),p,Q_{\ell}}\right).
\end{aligned}
\end{equation}

When $s_N \geq 1$, then arguments similar to those used in the
proof of Proposition~\ref{P:3.2} show that the expression on the
right side of \eqref{E:3.10} is majorized by
\[
 C^{\prime}\left(\Vert\,\bigl(A(x,D) - \lambda\,I_N\bigr)u\,\Vert^p_{(-s+1),p,\mathbb{R}^n} + \Vert\,u\,\Vert^p_{(t),p,\mathbb{R}^n}\right),
\]
where the constant $C^{\prime}$ has the same properties as $C$,
and hence it follows from \eqref{E:3.7}, \eqref{E:3.9},  and \eqref{E:3.10} that $u \in W^{(t+1)}_p(\mathbb{R}^n)$ and that the
inequality \eqref{E:3.8} holds.

Suppose next that for some $r, 1 < r \leq d$, $s_j = 0$ for $j > k_{r-1}$ and $s_j > 0$ for $j \leq k_{r-1}$.
Then
\[
H^{(-s+1)}_p(\mathbb{R}^n) = H^{(-s+1)_1}_p(\mathbb{R}^n) \times
W^1_p(\mathbb{R}^n)^{N-k_{r-1}},
\]
where $H^{(-s+1)_1}_p(\mathbb{R}^n) =
\prod_{j=1}^{k_{r-1}}H^{-s_j+1}_p(\mathbb{R}^n)$, and where now
$H^{(-s+1)}_p(\mathbb{R}^n), H^{(-s+1)_1}_p(\mathbb{}{R}^n)$ and
its adjoint space $W^{(s-1)_1}_{p^{\prime}}(\mathbb{R}^n) =
\prod_{j=1}^{k_{r-1}}W^{s_j-1}_{p^{\prime}}(\mathbb{R}^n)$, and
$W^{1}_p(\mathbb{R}^n)^{N-k_{r-1}}$ are  equipped with their
ordinary norms, $\Vert\,\cdot\,\Vert_{(-s+1),p,\mathbb{R}^n},
\Vert\,\cdot\,\Vert_{(-s+1)_1,p,\mathbb{R}^n},
\Vert\,\cdot\,\Vert_{(s-1)_1,p^{\prime},\mathbb{R}^n}$, and \break
$\Vert\,\cdot\,\Vert_{(1)_1,p,\mathbb{R}^n}$, respectively, which
are defined in a manner analogous to the way $W^{(\tau)}_p(G)$ and
$H^{(-s)}_p(G)$ were defined at the beginning of this section. For
our purposes it will be convenient to impose equivalent norms on
$W^{(s-1)_1}_{p^{\prime}}(\mathbb{R}^n)$ and
$W^{1}_p(\mathbb{R}^n)^{N-k_{r-1}}$ namely  \break
$\Vert\,\cdot\,\Vert^{\prime}_{(s-1)_1,p^{\prime},\mathbb{R}^n}$
and $\Vert\,\cdot\,\Vert^{\prime}_{(1)_1,p,\mathbb{R}^n}$,
respectively, where
$$\Vert\,u\,\Vert^{\prime}_{(s-1)_1,p^{\prime},\mathbb{R}^n}= 
\left(\sum_{j=1}^{k_{r-1}}\Vert\,u_j\,\Vert^p_{s_j-1,p^{\prime},\mathbb{R}^n}\right)^{1/p}$$
for $u \in W^{(s-1)_1}_{p^{\prime}}(\mathbb{R}^n)$ and
$$\Vert\,u\,\Vert^{\prime}_{(1)_1,p,\mathbb{R}^n} = 
\left(\sum_{j=k_{r-1}+1}^N\Vert\,u_j\,\Vert^p_{1,p,\mathbb{R}^n}\right)^{1/p}$$
for $u \in W^1_p(\mathbb{R}^n)^{N-k_{r-1}}$. Note that the norms
 $\Vert\,\cdot\,\Vert_{(s-1)_1,p^{\prime},\mathbb{R}^n}$ and
$\Vert\,\cdot\,\Vert^{\prime}_{(s-1)_1,p^{\prime},\mathbb{R}^n}$
induce equivalent norms on the adjoint space of
$W^{(s-1)_1}_{p^{\prime}}(\mathbb{R}^n)$,
$H^{(-s+1)_1}_p(\mathbb{R}^n)$, and hence if we denote by
$\Vert\,\cdot\,\Vert^{\prime}_{(-s+1)_1,p,\mathbb{R}^n}$ the norm
induced on \break   $H^{(-s+1)_1}_p(\mathbb{R}^n)$ by the norm
$\Vert\,\cdot\,\Vert^{\prime}_{(s-1)_1,p^{\prime},\mathbb{R}^n}$,
then $\Vert\,\cdot\,\Vert^{\prime}_{(-s+1)_1,p,\mathbb{R}^n}$ and
$\Vert\,\cdot\,\Vert_{(-s+1)_1,p,\mathbb{R}^n}$ are equivalent
norms on this space.

Let us now fix our attention upon $W^{(s-1)_1}_{p^{\prime}}(\mathbb{R}^n)$ and for $1 \leq j \leq k_{r-1}$
let $\tilde{\nu}_j$
denote the number of distinct multi-indices $\alpha$ for which $|\alpha| \leq s_j-1$.
Then to each $v = (v_1\ldots,v_{k_{r-1}})^T
 \in W^{(s-1)_1}_{p^{\prime}}(\mathbb{R}^n)$ we can associate the vector $$\mathcal{P}^{(1)}_{p^{\prime}}v
 =
 \left(\mathcal{P}^{(1)}_{1,p^{\prime}}v_1,\ldots,\mathcal{P}^{(1)}_{k_{r-1},p^{\prime}}v_{k_{r-1}} \right)
 \in L_{p^{\prime}}(\Omega)^{\sum_{j=1}^{k_{r-1}}\tilde{\nu}_j},$$ where for $1 \leq j \leq k_{r-1},
 \mathcal{P}^{(1)}_{j,p^{\prime}}v$
 denotes the $\tilde{\nu}_j$-th vector with components
 $\{\,D^{\alpha}v_j\,\}_{|\alpha| \leq s_j-1}$, and where the components of
  $\mathcal{P}^{(1)}_{j,p^{\prime}}$ are arranged in the same way as the entries of $\mathcal{P}_{j,p^{\prime}}$
  were arranged
 in the proof of Proposition~\ref{P:3.2}. Hence if we let $\mathcal{H}^{(1)}_{p^{\prime}}$ denote the closed subspace of
 $L_{p^{\prime}}^{\sum_{j=1}^{k_{r-1}}\tilde{\nu}_j}(\mathbb{R}^n)$ spanned by the vectors $\mathcal{P}^{(1)}_{p^{\prime}}v$ for
 $v \in W^{(s-1)_1}_{p^{\prime}}(\mathbb{R}^n)$ and if we equip $W^{(s-1)_1}_{p^{\prime}}(\mathbb{R}^n)$
 with the norm $\Vert\,\cdot\,\Vert_{(s-1)_1,p^{\prime},\mathbb{R}^n}^{\prime}$, then $\mathcal{H}^{(1)}_{p^{\prime}}$ is isometrically isomorphic to
 $W^{(s-1)_1}_{p^{\prime}}(\mathbb{R}^n)$. Hence it follows, as in the proof of Proposition~\ref{P:3.2}, that for each
 $u \in \mathcal{H}^{(-s+1)_1}_p(\mathbb{R}^n)$ there is an $f^{(1)}= \left(f^{(1)}_1,\ldots,f^{(1)}_{k_{r-1}}\right) \in L_p(\mathbb{R}^n)^{\sum_{j=1}^{k_{r-1}}\tilde{\nu}_j}$
 where $f^{(1)}_j = \left(f^{(1)}_{j,1},\ldots,f^{(1)}_{j,\tilde{\nu_j}}\right)$
 such that
  \begin{equation} \label{E:3.11}
 \Vert\,u\,\Vert^{\prime}_{(-s+1)_1,p,\mathbb{R}^n} = \Vert\,f^{(1)}\,\Vert_{\mathbb{R}^n},
 \end{equation}
 where $\Vert\,\cdot\,\Vert_{\mathbb{R}^n}$ denotes the norm in $L_p(\mathbb{R}^n)^{\sum_{j=1}^{k_{r-1}}\tilde{\nu}_j}$. Consequently if we
 fix our attention again upon \eqref{E:3.10} and let $P^{(1)}_p$ denote the operator projecting $H^{(-s+1)}_p(\mathbb{R}^n)$ onto $H^{(-s+1)_1}_p(\mathbb{R}^n)$ along
$W^1_p(\mathbb{R}^n)^{N-k_{r-1}}$, then there is an $f^{(1)} \in L_p(\mathbb{R}^n)^{\sum_{j=1}^{k_{r-1}}\tilde{\nu}_j}$ satisfying \eqref{E:3.11} with
$u$ there replaced by $P^{(1)}_p\bigl(A(x,D)- \lambda\,I_N\bigr)u$ such that for $\ell \geq 1$ we have
\begin{equation*}
 \begin{aligned}
   \Vert\,P^{(1)}_p\eta_{\ell}\left(A(x,D) - \lambda\,I_N\right)u\,&\Vert_{(-s+1)_1,p,Q_{\ell}}
   \leq \Vert\,P^{(1)}_p\eta_{\ell}\left(A(x,D) -
  \lambda\,I_N\right)u\,\Vert_{(-s+1)_1,p,\mathbb{R}^n}
  \\
  & \leq C\Vert\,P^{(1)}_p\eta_{\ell}\left(A(x,D) - \lambda\,I_N\right)u\,\Vert^{\prime}_{(-s+1)_1,p,\mathbb{R}^n} \\
  &  = C\,\textrm{sup}\,\bigl|\langle\,P^{(1)}_p\eta_{\ell}\left(A(x,D) - \lambda\,I_N\right)u,\zeta\,\rangle_{\mathbb{R}^n}\bigr|  \\
  & = C\,\textrm{sup}\,\bigl|\sum_{j=1}^{k_{r-1}}\sum_{k=1}^{\tilde{\nu}_j}\left(f^{(1)}_{j,k},D^{\alpha}\eta_{\ell}
  \zeta_j\right)_{\mathbb{R}^n}\bigr|
\leq C_1\Vert\,f^{(1)}\,\Vert_{Q_{\ell}},
   \end{aligned}
\end{equation*}
where in each case the supremum is over the set $\bigl\{\,\zeta
\in
C^{\infty}_0(\mathbb{R}^n)^{k_{r-1}}\,\bigl|\,\Vert\,\zeta\,\Vert_{(s-1)_1,p^{\prime},\mathbb{R}^n}
=1\,\bigr\}$, the constant $C_1$ has the same properties as the
constant $C$, $f^{(1)} \in
L_p(\mathbb{R}^n)^{\sum_{j=1}^{k_{r-1}}\tilde{\nu}_j}$ denotes the
norm preserving extension of $P^{(1)}_p\left(A(x,D) -
\lambda\,I_N\right)u$ to a linear functional over
$L_{p^{\prime}}(\mathbb{R}^n)^{\sum_{j=1}^{k_{r-1}}\tilde{\nu}_j}$,
and the terminology used here is analogous to that used in the
proof of Proposition \ref{P:3.2}. If we now let $f^{(2)} =
\left(f^{(2)}_{k_{r-1}+1},\ldots,f^{(2)}_N\right) \in
L_p(\mathbb{R}^n)^{(N-k_{r-1})(n+1)}$, where $f^{(2)}_j =
\left(f^{(2)}_{j,1},\ldots,f^{(2)}_{j,n+1}\right)$ and
$f^{(2)}_{j,k} = D_{k-1}u_j$, with $D_0u_j = u_j$, then we
conclude from the foregoing results that
 \begin{align}
   \sum_{\ell\geq1}\Vert\,\eta_{\ell}\bigl(A(x,D) - &\lambda\,I_N\,\bigr)u\,\Vert^p_{(-s+1),p,Q_{\ell}}   \notag
   \\
   & \leq C\sum_{\ell\geq1}\left(\sum_{j=1}^{k_{r-1}}\sum_{k=1}^{\tilde{\nu}_j}\Vert\,f^{(1)}_{j,k}\,\Vert^p_{0,p,Q_{\ell}} +
  \sum_{j=k_{r-1}+1}^N\sum_{k=1}^{n+1}\Vert\,f^{(2)}_{j,k}\,\Vert^p_{0,p,Q_{\ell}}\right) \notag   \\
    & \leq C^{\prime}\Vert\,\left(A(x,D) - \lambda\,I_N\right)u\,\Vert^p_{(-s+1),p,\mathbb{R}^n}, \notag
   \end{align}
where the constant $C^{\prime}$ has the same properties as the constant $C$. In light of these last inequalities,
 \eqref{E:3.9},  and \eqref{E:3.10}, we conclude that $u \in W^{(t+1)}_p(\mathbb{R}^n)$ and that the inequality \eqref{E:3.1} holds
 with $t$ and $-s$ replaced by $t+1$ and $-s+1$, respectively.

 Suppose next that $s_j = 0$ for $j = 1,\ldots,N$. Then $H^{(-s+1)}_p(\mathbb{R}^n) = W^1_p(\mathbb{R}^n)^N$, and for this case the proposition can be proved
 by arguing with $W^1_p(\mathbb{R}^n)^N$ as we argued with $W^1_p(\mathbb{R}^n)^{N-k_{r-1}}$ in the previous case.

 If $\ell_0 = 1$, then the proof of the proposition is complete. Otherwise we complete the proof by proceeding by induction.
\end{proof}

\section{Fredholm theory}  \label{S:4}
In this section we are going to use the results of
Section~\ref{S:3} to derive information pertaining to the Fredholm
theory for the Banach space operators induced by the spectral
problem \eqref{E:1.1}. Furthermore, when in the sequel we refer to
$W^{(\tau)}_p(G), \tau = t \; \textrm{or} \;s$ and $H^{(-s)}_p(G)$
as Banach spaces (see Section~\ref{S:3} for terminology), then it
is to be understood that we are equipping these spaces with their
ordinary norms. If $X$ and $Y$ are Banach spaces, then we shall
also use the notation $L\left(X,Y\right)$ to denote the space of
bounded linear operators mapping $X$ into $Y$ and equipped with
its usual norm.

 Next  let      $A_p$ denote the operator on $H^{(-s)}_p(\mathbb{R}^n)$ that acts like $A(x,D)$
 and has domain $D(A_p) =  W^{(t)}_p(\mathbb{R}^n)$.

We note for later use that that if we suppose that the hypotheses
of Proposition~\ref{P:3.2} hold, then
\[
 \Vert\,u\,\Vert_{(t),p,\mathbb{R}^n} \leq C\left(\Vert\,A_p\,u\,\Vert_{(-s),p,\mathbb{R}^n} +
 \Vert\,u\,\Vert_{(-s),p,\mathbb{R}^n}\right) \quad \textrm{for} \quad u \in D(A_p),
\]
where the constant  $C$ does not depend upon $u$. Thus we conclude that the operator
$A_p: D(A_p) \to H^{(-s)}_p(\mathbb{R}^n)$ is closed. Note also
  from Proposition~\ref{P:3.1} that  $A_p: D(A_p) \to H^{(-s)}_p(\mathbb{R}^n)$ is bounded.

Referring to \cite[pp.~242--243]{Ka} for terminology, we now have
the following result.
\begin{theorem} \label{T:4.1}
 Suppose that the spectral problem \eqref{E:1.1} is  weakly smooth
 and
 parameter-elliptic in $\mathcal{L}$.
 Then there exists
 the number $\lambda^{\dagger} \in \mathbb{R}_+$,  where $\lambda^{\dagger}$ does not depend upon $p$,
 such that
 $A_p - \lambda\,I_N \in L\left(D(A_p),
   H^{(-s)}_p(\mathbb{R}^n)\right)$ and is Fredholm
for   $\lambda \in \mathcal{L}^{\dagger} =
  \bigl\{\,\lambda \in \mathcal{L} \bigl| \, |\lambda| \geq \lambda^{\dagger}\,\bigr\}$. Furthermore, there exists the number $\lambda^0(p) \geq \lambda^{\dagger}$ such that
      $\lambda$ belongs to the resolvent set of $A_p$ for $\lambda \in \mathcal{L}^{\dagger}_p = \bigl\{\,\lambda \in \mathcal{L}^{\dagger}\,\bigl|
   |\lambda| \geq \lambda^0(p)\,\bigr\}$, and hence    index\,$(A_p - \lambda\,I_N) = 0$ for $\lambda \in \mathcal{L}^{\dagger}$.
   Finally, that part of the spectrum of $A_p$ which is contained in $\mathcal{L}^{\dagger}$ consists solely of most a finite number of distinct eigenvalues,
   with each being of finite algebraic multiplicity.
      \end{theorem}
\begin{proof}
We know from the proof of Proposition~\ref{P:3.4} that for $x,\xi
\in \mathbb{R}^n$ and $\lambda \in \mathcal{L}$ with $|\lambda|
\geq \lambda_0$,
\[
 |\textrm{det}\,\left(\textrm{\r{A}}(x,\xi) - \lambda\,I_N\right)| \geq C_0\prod_{j=1}^N\langle\,\xi,\lambda\,\rangle_j^{m_j},
\]
where the constant $C_0$ does not depend upon $x, \xi,$ and $\lambda$. Furthermore, it is not difficult to verify that for these values of
$x, \xi,$ and $\lambda$,
\[
 |\textrm{det}\left(\textrm{\r{A}}(x,\xi) - \lambda\,I_N\right) - \textrm{det}\left(A(x,\xi) - \lambda\,I_N\right)| \leq C|\lambda|^{-1/m_1}
 \prod_{j=1}^N\langle\,\xi,\lambda\,\rangle_j^{m_j},
\]
where the constant $C$ does not depend upon $x, \xi,$ and $\lambda$. Hence it follows that we can choose the number $\lambda^{\dagger} \in \mathbb{R}_+$ such that
\[
 |\textrm{det}\left(A(x,\xi) - \lambda\,I_N\right)| \geq C_0/2\prod_{j=1}^N\langle\,\xi,\lambda\,\rangle_j^{m_j}
\]
for $x,\xi \in \mathbb{R}^n$ and $\lambda \in \mathcal{L}$ with
$|\lambda| \geq \lambda^{\dagger}$. Thus all but the final
assertions of the proposition now follow from this last result,
\cite{R}, Proposition  \ref{P:3.3}, from what was said in the text
preceding the statement of this theorem, and from the fact that
$\mathcal{L}^{\dagger}$ is contained in a component of the
Fredholm domain of $A_p$ (see \cite[pp.~242--243]{Ka}).

Turning now to the final assertions of
the proposition,
  let $\lambda_1, \lambda_2
\in \mathcal{L}^{\dagger}$ with $|\lambda_1| \geq \lambda^0(p)$ and $|\lambda_2| < \lambda^0(p)$.
Then there is a polygonal arc $\gamma = \bigl\{\,\gamma(t), 0 \leq t \leq 1\,\bigr\}$
joining $\lambda_1$ to $\lambda_2$ and
 lying entirely in $\mathcal{L}^{\dagger}$ such that $\gamma(0) = \lambda_1$ and $\gamma(1) = \lambda_2$.
  But this implies that either $N(t) = \textrm{dim\,ker}\,\left(A_p - \gamma(t)\,I_N\right) = 0$  for $0 \leq t < 1$ or there is a $\tau_1, 0 < \tau_1 < 1$ such that
 $N(t) = 0$ for $0 \leq t < \tau_1, N(\tau_1) > 0$.
However for either case we know from \cite[Theorem~5.3.1,
p.~241]{Ka} that for some $\epsilon > 0$, dim\,ker$(A_p -
\lambda\,I_N) = 0$ for $0 < |\lambda - \lambda_2| < \epsilon$ if
the first case occurs and for $0 < |\lambda - \gamma(\tau_1)|
 < \epsilon$ if the second case occurs. Furthermore, if the second case occurs, then either $N(t) = 0$ for $0 \leq t < 1$ except for $t = \tau_1$,
 or there is a $\tau_2, 0 < \tau_1 < \tau_2 < 1$, such that $N(t) = 0$ for $ 0 \leq t \leq \tau_2$  except for $t = \tau_j, j = 1,2$.
 But as before, we know that in either case there is an $\epsilon > 0$ such that dim\,ker$\left(A_p - \lambda\,I_N\right)) = 0$
 for $0 < |\lambda - \lambda_2| < 0$ if the first case holds  and for $0 < |\lambda - \gamma(\tau_2)| < \epsilon$ if the second case holds.
 Carrying on in this manner we finally arrive at the situation where either there is a finite sequence
 $\{\,\tau_j\,\}_1^k, 0 < \tau_1< \cdots < \tau_k < 1$ such that $N(t) = 0$ for $0 \leq t <1$ except for $ t = \tau_j, j = 1,\ldots,k$,  or
 there is an infinite sequence $\{\,\tau_j\,\}_1^{\infty}, 0 < \tau_1  < \tau_2 < \cdots < 1$, such that $N(t) = 0$ for $0 \leq t  < \tau$
 except for $t = \tau_j, j \geq 1$,
 where $\tau = \lim_{j \to \infty} \tau_j$. Then we see from \cite{Ka} that the second case is not possible, while for the first case  there is an $\epsilon > 0$
  such that dim\,ker$\left(A_p - \lambda\,I_N\right) = 0$ for $0 < |\lambda - \lambda_2| < \epsilon$.
 Thus we have shown that if $\mu \in \mathcal{L}^{\dagger}$ and dim\,ker$\left(A_p - \mu\,I_N\right) > 0$, then there is an $\epsilon > 0$
 such that dim\,ker$(A_p - \lambda\,I_N) = 0$ for $0 < |\lambda - \mu| < \epsilon$.
 The final assertions of the theorem follows from this last result, and  \cite{Ka},  which concludes the
 proof.
 \end{proof}

   We are now going to investigate how the eigenvalues, if any, of $A_p$ which lie in $\mathcal{L}^{\dagger}$ vary with $p$.

 \begin{theorem} \label{T:4.2}
  Suppose that the spectral problem \eqref{E:1.1} is weakly smooth
  and
  parameter-elliptic in $\mathcal{L}$.
    Suppose in addition that $a^{jk}_{\alpha} \in C^{s_j+n}(\overline{\mathbb{R}^n})$
  for $|\alpha| \leq s_j+t_k, 1 \leq j,k \leq N$.
    Lastly suppose
   that
   $\lambda_1 \in \mathcal{L}^{\dagger}$ is an eigenvalue of $A_p$ and $u^{(1)}$ a corresponding eigenvector. Then $\lambda_1$
   is an eigenvalue and $u^{(1)} \in W^{(t)}_q(\mathbb{R}^n)$ a corresponding eigenvector of $A_q$ for every $q$ satisfying $p < q < \infty$.
   Consequently  ker$\,\left(A_p - \lambda_1\,I_N\right) \subset$   ker$\,\left(A_q - \lambda_1\,I_N\right)$.
 \end{theorem}
\begin{proof}
As a consequence of Proposition~\ref{P:3.6} we see that $u^{(1)}
\in W^{(t+n)}_p(\mathbb{R}^n)$. Hence it follows from the Sobolev
embedding theorem (see \cite[Theorem~5.4, p.~97]{Ad}) that
$u^{(1)} \in W^{(t)}_q(\mathbb{R}^n)$ for every $q$ satisfying $p
< q < \infty$, and all the assertions of the theorem are immediate
consequences of this fact.
\end{proof}

Under some further restrictions  Theorem~\ref{T:4.2} can be
improved.
\begin{theorem} \label{T:4.3}
  Suppose that the hypotheses of Theorem  \ref{T:4.2} hold with $s_j = 0$ and $t_j$ even for $j = 1,\ldots,N$.
      Suppose in addition that for $1 \leq j,k \leq N$,
    $a^{jk}_{\alpha} \in C^{|\alpha|+n}(\overline{\mathbb{R}^n}) \cap C^{|\alpha|+1,0}(\overline{\mathbb{R}^n})$
    for $|\alpha| \leq t_k$ with $a^{jk}_{\alpha}(x) = 0$ for $x \in \mathbb{R}^n$ if $j > k, t_j < t_k$, and $|\alpha| \geq t_j$ and also if
    $j < k, t_j > t_k$, and $|\alpha| = t_k$.
    Then the eigenvalues of $A_p$ lying in $\mathcal{L}^{\dagger}$,
    as well as their geometric and algebraic multiplicities, are the same for all these values of p.
    \end{theorem}

     In order to prove the theorem, a preliminary result is required.
 \begin{proposition} \label{P:4.4}
 Suppose that the hypotheses of Theorem \ref{T:4.3} hold and let $1  < q < p$. Also let $\lambda_1 \in \mathcal{L}^{\dagger}$ and put
 $n_p = \textrm{dim\,ker}\,(A_p - \lambda_1\,I_N), d_p =
 \textrm{dim\,coker}\,(A_p - \lambda_1\,I_N)$. Then \\
 \indent(1) ker\,$(A_p - \lambda_1\,I_N) \subset$ ker\,$(A_q - \lambda_1\,I_N)$ if $0 \leq d_p \leq n_p < \infty$; \\
 \indent (2) ker\,$(A_p - \lambda_1\,I_N) \subset$ ker\,$(A_q - \lambda_1\,I_N)$ if $0 \leq n_p < d_p < \infty$.
 \end{proposition}
\begin{proof}
 To begin with let us introduce the multiplication operators $T_p(\tau)$ and $S_p(\tau)$, $\tau \geq 0$, acting on $D(A_p)$ and $L_p(\mathbb{R}^n)^N$
($= H^{(-s)}_p(\mathbb{R}^n)$), respectively,
 where $T_p(\tau)u(x) = e^{-\tau\,\langle\,x\,\rangle}u(x)$ for $ u \in D(A_p)$, $S_p(\tau)v(x) = e^{-\tau\,\langle\,x\,\rangle}v(x)$ for $v \in L_p(\mathbb{R}^n)^N$,
 where
 $\langle\,x\,\rangle = (1+|x|^2)^{1/2}$.

 Let $ u \in D(A_p)$, denote by $u_k$ its $k$-th component, and let $\alpha$ denote a multi-index such that $|\alpha| \leq t_k$.
 Then it follows from the Leibnitz formula that
 \[
  D^{\alpha}e^{\pm\tau\,\langle\,x\,\rangle}u_k(x) = e^{\pm\tau\,\langle\,x\,\rangle}D^{\alpha}u_k(x) + \sum_{ \substack{ \beta \leq \alpha\\
  |\beta| > 0}}\binom{\alpha}{\beta}\left(D^{\beta}e^{\pm\tau\,\langle\,x\,\rangle)}\right)D^{\alpha - \beta}u_k(x),
  \]
while direct calculations show that
 \begin{equation*}
 D^{\beta}e^{\pm\tau\,\langle\,x\,\rangle} = e^{\pm\tau\langle\,x\,\rangle}\sum\,^{\prime}\prod\,^{\prime}
 \left(\pm\tau\,D^{\gamma_j}\langle\,x\,\rangle\right)
  \end{equation*}
and that
 \begin{equation*}
 |D^{\gamma_j}\langle\,x\,\rangle| \leq  C(\gamma_j)\langle\,x\,\rangle^{1-|\gamma_j|},
 \end{equation*}
 where $\prod^{\prime}$ indicates that the product is taken over a set of distinct multi-indices $\gamma_j$ for which $|\gamma_j| > 0$ and
 $\sum_{j}|\gamma_j| = |\beta|$,
 $\sum^{\prime}$ indicates that the sum is taken over all such sets, and $C(\gamma_j)$ denotes a constant depending upon $\gamma_j$.
 In light of these facts it is not difficult to deduce that for $\tau \in [0,\infty)$, $T_p(\tau) \in L\left(D(A_p),D(A_p)\right)$ and
 $S(\tau) \in L\left(L_p(\mathbb{R}^n)^N,L_p(\mathbb{R}^n)^N\right)$
  are $C_0$ semigroups such that for each $\tau, T(\tau)$ and $S(\tau)$ are injective,while $S(\tau)^{-1}A_pT(\tau) \in L\left(D(A_p),L_p(\mathbb{R}^n)^N\right)$ and that
  $S(\tau)^{-1}A_pT(\tau)$ converges in norm to $A_p$ as $\tau \to 0$.

  Let us next fix our attention upon the assertions of the proposition, and to begin with let us consider assertion (1) with $n_p > 0$ and $d_p = 0$.
  Then it follows from \cite[Theorem~5.22, p.~236]{Ka} that for some $\tau_0 > 0, V = S(\tau_0)^{-1}(A_p - \lambda_1\,I_N)T(\tau_0)$ is semi-Fredholm
  with dim\,ker\,$V = n_p$, and dim\,coker\,$V = 0$. Hence by arguing as in \cite{R} and applying Holder's inequality,  we readly deduce that
  ker$(A_p - \lambda_1\,I_N)  \subset$ ker$(A_q- \lambda_1\,I_N)$. Obviously the same result is true if we assume instead that $n_p = 0$.


  Finally assertions (1), with $d_p > 0$, as well as assertion (2), can be proved by modifying the above arguments as in \cite{R}.
\end{proof}

\noindent\textit{Proof of Theorem 4.3.} Referring to
Proposition~\ref{P:4.4} for notation,
 let us firstly fix our attention upon the case  $1 < q < p$ and prove all but the final assertions
  of the theorem for this case. Accordingly, to begin with, let us observe
  from Theorem~\ref{T:4.2} and Proposition~\ref{P:4.4} that ker$(A_q - \lambda_1\,I_N) =$ ker$(A_p - \lambda_1\,I_N)$.

We now turn our attention to cokernels. Then it is not difficult to show
that under our assumptions
 the spectral problem formally adjoint to the spectral problem \eqref{E:1.1},
\begin{equation} \label{E:4.1}
A^{\prime}(x,D)u(x) - \overline{\lambda_1}u(x) = f(x) \quad
\textrm{for} \quad x \in \mathbb{R}^n,
\end{equation}
is well defined. Here $A^{\prime}(x,D) =
\left(A^{\prime}_{jk}(x,D\right)_{j,k=1}^N$ is the formal adjoint
of $A(x,D)$, the $A^{\prime}_{jk}(x,D)$ are linear differential
operators defined on $\mathbb{R}^n$ of order not exceeding
$s^{\prime}_j + t^{\prime}_k$ where $s^{\prime}_j = 0,
t^{\prime}_k = t_k$ for $1 \leq j,k \leq N$. Note also that the
top order operators $\textrm{\r{A}}(x,D)$ and
$\textrm{\r{A}}^{\prime}(x,D)$ are block diagonal. Then it is  not
difficult to verify that the analogues of
Propositions~\ref{P:3.1}--\ref{P:3.3} for the spectral problem
\eqref{E:4.1} also hold and we can take the constant
$\lambda^{\sharp}$ of Proposition~\ref{P:3.1} to remain the same.
We henceforth let $A^{\prime}_{p^{\prime}}$ denote the operator on
$L_{p^{\prime}}(\mathbb{R}^n)^N$ that acts like $A^{\prime}(x,D)$
and has domain $D(A^{\prime}_{p^{\prime}}) =
W^{(t)}_{p^{\prime}}(\mathbb{R}^n)$; and we can readily verify
that the analogues of Theorems~\ref{T:4.1}, \ref{T:4.2} and
Proposition~\ref{P:4.4} hold for $A^{\prime}_{p^{\prime}}$. Hence
we can now argue as above with $A^{\prime}_{q^{\prime}} -
\overline{\lambda_1}\,I_N$ and $A^{\prime}_{p^{\prime}} -
\overline{\lambda_1}\,I_N$ in place of $A_p - \lambda_1\,I_N$ and
$A_q - \lambda_1\,I_N$, respectively, to deduce that coker\,$(A_q
- \lambda_1\,I_N)$ = coker\,$(A_p - \lambda_1\,I_N)$. This proves
all but the final assertions of the theorem for the case $q < p$,
and the analogous result for the case $q > p$ can be proved in a
similar fashion (i.e., by arguing with the adjoint operators).

Turning to the final assertions of the theorem, let $\lambda_1 \in \mathcal{L}^{\dagger}$ be an eigenvalue
of $A_p$ and $u^{(1)}$
a corresponding eigenvector. Then it is clear from what has already been proved that we need only prove that
the algebraic multiplicity
 of $\lambda_1$ is the same for all $p$. Accordingly, let us suppose firstly that $q > p$ and
 let $\bigl\{\,u^{(1,j)}\,\bigr\}_{j=0}^{m-1}$
 be a chain of length $m$ consisting of the eigenvector $u^{(1,0)} = u^{(1)}$ and the associated vectors
 $\bigl\{\,u^{(1j)}\,\bigr\}_{j=1}^{m-1}$ corresponding to the eigenvalue $\lambda_1$ of $A_p$
 (see \cite[pp.~60--61]{M}).
 Thus it follows from Proposition~\ref{P:3.6} and the Sobolev embedding theorem
(see \cite[Theorem~5.4, p.~97]{Ad}) that
$\bigl\{\,u^{(1,j)}\,\bigr\}_{j=0}^{m-1}$ is a chain of length $m$
consisting of the eigenvector $u^{(1)} = u^{(1,0)}$ and the
associated vectors $\bigl\{\,u^{(1,j)}\,\bigr\}_{j=1}^{m-1}$
corresponding to the eigenvalue $\lambda_1$ of $A_q$. We conclude
immediately that the algebraic multiplicity of $\lambda_1$ as an
eigenvalue of $A_p$ does not exceed the algebraic multiplicity of
$\lambda_1$ as an eigenvalue of $A_q$. On the other hand we can
appeal to the analogue of Proposition~\ref{P:3.6} for the spectral
problem \eqref{E:4.1} to show that the algebraic multiplicity of
$\overline{\lambda_1}$ as an eigenvalue of
$A^{\prime}_{q^{\prime}}$ does not exceed the algebraic
multiplicity of $\overline{\lambda_1}$ as an eigenvalue of
$A^{\prime}_{p^{\prime}}$. Hence we conclude that the algebraic
multiplicity of $\lambda_1$ as an eigenvalue of $A_p$ and of $A_q$
are the same (see \cite[p.~184]{Ka}). Since similar arguments give
the same result for $q < p$, the proof of the theorem is complete.
\hfill{$\square$}

 \section{An example} \label{S:5}
 In this section we fix our attention upon the spectral problem \eqref{E:1.1} with $A(x,D)$ a $2 \times 2$ matrix operator
 whose entries $A_{jk}(x,D)$ are linear differential operators defined on $\mathbb{R}^n$ of order not exceeding $s_j + t_k$, where $s_j = 0$ for $j = 1,2$,
 and $t_1 = 4, t_2 = 2$. To be more precise we now take
 \begin{equation} \label{E:5.1}
  A(x,D) = A_0(x,D) - cI_2 + \left(\tilde{A}_{jk}(x,D)\right)^2_{j,k=1},
 \end{equation}
where $A_0(x,D)$ = diag\,$\left(\Delta^2,-\Delta\right)$, $\Delta$ denotes the Laplacian in $\mathbb{R}^n, c \in \overline{\mathbb{R}_+}$ denotes a constant,
$\tilde{A}_{jk}(x,D) = \sum_{|\alpha| \leq \sigma_{jk}}a^{jk}_{\alpha}(x)D^{\alpha}$ with $\sigma_{jk} = 1$ if $j \ne k, \sigma_{11} = 3$,  $\sigma_{22} = 1$,
and the $a^{jk}_{\alpha} \in C_0^{\infty}(\mathbb{R}^n)$.
 Then with $A(x,D)$ given by \eqref{E:5.1}, let us now investigate the spectral theory connected with the problem \eqref{E:1.1}.

 Accordingly, let us fix an $\epsilon$ satisfying $0 < \epsilon < \pi/2$ and let $\mathcal{L}$ denote the sector in the complex plane with vertex at the origin
 determined by the inequalities $\epsilon \leq  \arg\,\lambda \leq 2\pi-\epsilon$. Then we can readily verify that the hypotheses of Theorem \ref{T:4.3}
 concerning the spectral problem \eqref{E:1.1} are satisfied, and hence if we let $A_p$ denote the operator on $L_p(\mathbb{R}^n)^2$
 that acts like $A(x,D)$ and has domain $D(A_p) = W^{(t)}_p(\mathbb{R}^n)$, the all the assertions
 of Theorems~\ref{T:4.1} and \ref{T:4.3} hold.

 In order to derive more information concerning the spectral properties of $A_p$, we are now going to fix our attention upon the case $p = 2$,
 and let $A^{(1)}_2$ (resp. $A^{(2)}_2$) denote the operator on $L_2(\mathbb{R}^n)$ that acts like $\Delta^2$
 (resp. $-\Delta$) and with domain
 $D(A^{(1)}_2) = W^4_2(\mathbb{R}^n)$ (resp. $D(A^{(2)}_2) = W^2_2(\mathbb{R}^n)$.
 Then direct calculations show that $A^{(1)}_2$ is a symmetric ope\-ra\-tor on
 $L_2(\mathbb{R}^n)$ whose numerical range is contained in the interval $[0,\infty)$.
 Furthermore, since it is shown in \cite{F} that the analogue of Theorem \ref{P:3.3}
 holds when the spectral problem \eqref{E:1.1} is replaced by a scalar spectral problem,
 we conclude from these results that
 $A^{(1)}_2$ is a selfadjoint operator on $L_2(\mathbb{R}^n)$ whose spectrum is contained in the interval $[0,\infty)$.
 Similarly we can show that
 $A^{(2)}_2$ is a selfadjoint operator on $L_2(\mathbb{R}^n)$
 whose spectrum is contained in $[0,\infty)$, and furthermore, we  know from \cite[p.~416]{B} (see also \cite[p.~158]{Ti}
 for the case $n=2$)
 that its spectrum is precisely $[0,\infty)$.

 Thus we have shown that if we let $A_{0,2}$ = diag\,$(A^{(1)}_2,A^{(2)}_2\,)$, then $A_{0,2} - c\,I_2$ is a selfadjoint
 operator on $L_2(\mathbb{R}^n)^2$
 with domain $D(A_{0,2}) =
 W^{(t)}_2(\mathbb{R}^n)$ and whose spectrum is precisely $[-c,\infty)$. Furthermore, we can appeal
 to Proposition~\ref{P:3.2}
  to show that $A_2 -(A_{0,2} - c\,I_2)$ is relatively compact with respect to
 $A_{0,2} - c\,I_2$, and hence it follows from \cite[Theorem~5.35, p.~244]{Ka} that $A_2$ is a closed
 operator on $L_2(\mathbb{R}^n)^2$
 with essential spectrum $[-c,\infty)$ and with semi-Fredholm domain $\mathbb{C} \setminus [-c,\infty)$.
 Hence referring again to Theorem~\ref{T:4.1},
 it is clear that we must have $\lambda^{\dagger} > c$, while it also follows from the arguments used in the
 proof of the final assertion of that theorem
 that the  Fredholm domain of $A_2$  is  precisely $\mathcal{C} \setminus[-c,\infty)$, that $\mathcal{C}\setminus[-c,\infty)$ consists of only one component,
 and that index\,$(A_2 - \lambda\,I_2) = 0$ \,for \,$\lambda$ lying in this component.



\end{document}